\documentclass[11pt]{amsart}

\usepackage{amsmath,amsfonts,amssymb,latexsym}

\usepackage{amsmath}
\usepackage{amssymb}
\usepackage{amsthm}    
\usepackage{amscd}     
\usepackage[all]{xy}  
\usepackage{remreset}   
\usepackage{latexsym}     

\def\NN{\mathbb N}
\def\ZZ{\mathbb Z}
\def\QQ{\mathbb Q}
\def\RR{\mathbb R}
\def\CC{\mathbb C}

\def\PP{\mathbb P}

\def\bcp{\mathbb C \mathbb P}
\def\cpb{\overline{\mathbb C \mathbb P^2}}

\def\OO{\cal O}

\def\cal{\mathcal} 
\def\wt{\widetilde}

\theoremstyle{plain} \numberwithin{equation}{section}
\newtheorem{main}{Theorem}

\newtheorem{thm}{Theorem}[section]

\newtheorem{prop}[thm]{Proposition}

\newtheorem{lem}[thm]{Lemma}

\newtheorem{constr}[thm]{Construction}

\theoremstyle{definition}
\newtheorem{remark}{Remark}[section]

\newtheorem{defn}[thm]{Definition}

\newtheorem{rmk}[remark]{Remark}

\DeclareMathOperator{\card}{card}

\makeatother

\begin{document}

\title[On Einstein metrics on 4-manifolds]{On Einstein Metrics on 4-Manifolds with Finite Cyclic Fundamental Group}
\author{Ioana {\c S}uvaina}

%


\begin{abstract}
 The existence or non-existence of Einstein metrics on 4-manifolds with non-trivial fundamental group and the relation with the underlying differential structure are analyzed.
For most points $(n,m)$ in a large region of the integer lattice, the manifold $n\bcp^2\#m \cpb$
is shown to admit infinitely many inequivalent free actions of finite cyclic groups and there are no
 Einstein metrics which are invariant under any of these actions.
   The main tools are
 Seiberg-Witten theory, cyclic branched coverings of complex surfaces and symplectic surgeries.

\end{abstract}
\maketitle 


\section{Introduction} \label{introduction}

The classical obstructions to the existence of an Einstein metric on a four-manifold are topological. If $(M,g)$ is a smooth compact oriented four-manifold endowed with an Einstein metric $g$ then its Euler characteristic, $\chi(M),$ must be non-negative and it satisfies the Hitchin-Thorpe inequality:

\begin{equation}\label{hti}
	(2\chi \pm 3\tau ) (M) \geq 0
\end{equation}
where $\tau(M)$ is the signature of $M.$ 
Moreover, the equality is obtained only for manifolds which are covered by a 
flat $4$-torus, the K3 surface with a hyperk\"ahler metric or these with the reverse orientations.

Using the Seiberg-Witten equations, LeBrun \cite{lno} found obstructions to the existence 
of Einstein metrics on a large class of manifolds for which the topological obstructions are satisfied. 
He later refined his arguments to improve the numerical condition in the obstruction. 
The main obstruction theorem that we are going to use in this paper is the following:
\begin{thm} {\em \cite{lric}} \label{E-obstruction}
Let $X$ be a compact oriented 4-manifold with a non-trivial
 Seiberg-Witten invariant and with $(2\chi +3\tau)(X)>0.$ Then\\
 $$ M=X\# k\overline{\bcp^2}$$
 does not admit Einstein metrics if
 $k\geq \frac13(2\chi+3\tau)(X).$
\end{thm}

These obstructions provided the means of exhibiting the strong dependence of the existence 
of Einstein metrics on the differential structure of the underlying four-manifold. 
Examples of pairs of homeomorphic but not diffeomorphic simply connected manifolds 
such that one manifold admits an Einstein metric while the other does not
 were first found by Kotschick \cite{kot}. 
Later on, other examples were constructed by LeBrun, Kotschick, 
 the two in joint work with collaborators \cite{i-lb, i-lb-o,bra-kot} 
and by others \cite{ra-su,irs} 
in order to exhibit the existence or non existence of Einstein metrics for infinitely many homeomorphic
non diffeomorphic smooth structures on the same topological space.

The explicit examples which are emphasized in the literature are of simply-connected manifolds. 
In this paper we will analyze the case of non-simply connected manifolds. We also study some 
immediate consequences on the obstruction of invariant Einstein metrics on the standard smooth 
structure on the connected 
sum of complex projective planes and complex projective planes with reversed orientation. 

Our first theorem is about manifolds with arbitrary finite cyclic fundamental group:
\begin{main}\label{main}
For any finite cyclic group $\ZZ/d\ZZ,d>1,$ there exist infinitely
 many families of compact oriented smooth 4-manifolds
 $( Z_i ,\{M_{i,j} \}_{j \in \NN})_{i \in \NN},$ all having fundamental group $\ZZ/d\ZZ$ and satisfying:
\begin{itemize}
\item [1.] For $i$ fixed, any two manifolds in $\{Z_i,M_{i,j}\}$ are
 homeomorphic but not diffeomorphic to each other;
\item [2.] $Z_i$ admits an Einstein metric, while no
 $M_{i,j}$ admits an Einstein metric.
\end{itemize}
Moreover their universal covers, $\widetilde{Z_i}$ and $ \widetilde{M_{i,j}},$ satisfy:
\begin{itemize}
\item [3.] $\widetilde{M_{i,j}}$ is diffeomorphic to $ n\bcp^2 \# m\overline{\bcp}^2$,
 where $n=b_2^+(\widetilde{Z_i})$ and\linebreak
 $m=b_2^-(\widetilde{Z_i})$;
\item [4.] $\widetilde{Z_i} $ and $\widetilde{M_{i,j}}$ are not diffeomorphic, but become
 diffeomorphic after taking the connected sum with
 one copy of $\bcp^2$.
 \end{itemize}
\end{main} 

This theorem generalizes the results on simply connected manifolds in \cite[Theorem 1.3]{kot} and 
\cite[Theorem 2]{bra-kot}.

In \cite{i-lb,bra-kot}, it was proved that on $ n\bcp^2 \# m\overline{\bcp}^2$   there are infinitely many distinct non standard structures for which no Einstein metric exists, for suitable $n,m$. 
But there is not much known about the properties of the metrics on standard smooth structure. 
Using  constructions similar to the ones
 in the proof of Theorem \ref{main}, we  study 
 Einstein metrics which are invariant under the actions of
 $\ZZ/d\ZZ.$ Two of the examples with the
 smallest topological invariants are the following: 

\begin{prop} \label{explicit-2}
\begin{itemize}
\item[1.] There exists a smooth, orientation preserving,  involution $\sigma$ acting freely 
on the manifold  $15\bcp^2 \# 77 \overline{\bcp^2}$ such that
 $15 \bcp^2 \# 77 \overline{\bcp^2}$ does not admit a $\sigma$-invariant
 Einstein metric.
\item[2.] There exists a free, smooth, orientation preserving action of $\ZZ/3\ZZ$ on 
$23\bcp^2 \# 116 \overline{\bcp^2}$,  such that
 $23 \bcp^2 \# 116 \overline{\bcp^2}$ does not admit a
 $\ZZ/3\ZZ$-invariant Einstein metric.
\end{itemize}
\end{prop}

As we are discussing  the geography of Einstein $4$-manifolds, it is convenient to
consider the following topological invariant $2\chi+3\tau.$ This is equal to the Chern number $c_1^2$  when
the manifold supports an almost complex structure.  To simplify the notations, we will denote $2\chi+3\tau$ by $c_1^2$ for any $4$-manifold. 
Then the Hitchin-Thorpe inequality for Einstein manifolds requires that $c_1^2\geq0,$ 
with equality only in the few, known, special cases.
A second topological invariant which we  use is the Todd genus $\chi_h=\frac{\chi+\tau}{4}.$ 
This invariant is a positive integer for manifolds 
which have finite fundamental group, and $b_2^+$ odd, which is the case for 4-manifolds supporting 
an almost-complex structure, 
or the ones constructed by simple surgeries as in our paper.
When the manifold supports an integrable complex structure, $\chi_h$  computes
 the holomorphic Euler characteristic.
The Hitchin-Thorpe inequality
$2\chi-3\tau\geq0$ is equivalent to $\displaystyle c_1^2\leq 9.6 \chi_h.$
For both examples in Proposition \ref{explicit-2},  
the quotient manifolds have the following invariants: 
$c_1^2=1$ and $\chi_h= 4.$ 

\begin{main}\label{infit-act+metr}
 For any integer $d\geq2$ and any $\epsilon>0$ there exists an $N(\epsilon)>0$ 
 such that for any integer lattice point $(x,y)$ satisfying:
\begin{enumerate}
\item[1.] free action admissibility condition: $d| x, d| y,$
\item[2.] either $0<y<(6-\epsilon)x-N(\epsilon)$   or $ (9+\frac\epsilon4)x+\frac12N(\epsilon)<y<9.6x$ and $y$ is even,
\end{enumerate}
there exist infinitely many inequivalent free $\ZZ/d\ZZ-$actions on the standard smooth manifold $X$
with $\chi_h(X)=x$ and $c_1^2(X)=y.$
Moreover, there is no Einstein metric on $X$ invariant under any of the $\ZZ/d\ZZ-$actions.
\end{main}

The manifolds constructed in the theorem are diffeomorphic to $(2x-1)\bcp^2\#(10x-y-1)\cpb.$ 
The $\ZZ/d\ZZ-$actions are such that the quotient manifolds are homeomorphic, 
but not diffeomorphic. 
The differential structures are distinguished by their Seiberg-Witten invariants. 
As a consequence of the non-triviality of these invariants any invariant 
constant scalar curvature metric on $X$ must have non-positive constant.
One important ingredient in the proof of the above result is the geography of 
almost completely decomposable symplectic manifolds, due to Braungardt and Kotschick \cite{bra-kot}.

\begin{rmk}
For completeness, we mention the special when the manifolds are $\bcp^2\#m\cpb, 0\leq m\leq 8.$ 
The existence of Einstein metrics is unobstructed by the Hitchin-Thorpe inequality,
 and in fact the standard smooth structure admits an Einstein metric of positive scalar curvature. 
This was given by the Fubini-Study metric for $m=0$, by Page \cite{page} for $m=1$, 
by Chen-LeBrun-Weber \cite{c-l-w} for $m=2$ and by Siu \cite{siu} or Tian-Yau \cite{tiya} for $3\leq m\leq8.$ 
For $M=\bcp^2\#m\cpb, ~5\leq m\leq 8$ the author and R. R\u asdeaconu \cite{ra-su} showed that $M$
supports  a nonstandard smooth structure which admits an Einstein metric of  negative sign, 
 as well as infinitely many smooth structures which do not admit Einstein metrics.
\end{rmk}

\begin{rmk}
The result in Theorem \ref{infit-act+metr} is stated  for finite cyclic groups.
However, it also holds  for groups acting freely on the 3-dimensional sphere 
or for direct sums of the above groups. 
\end{rmk}

In general, the existence or non-existence of Einstein metrics on an arbitrary manifold is hard to prove. 
On the remaining part of this paper we emphasize the non-existence of
 Einstein metrics on non-simply connected manifolds.

\begin{main}\label{non-spin+arb}
Let $G$ be any finitely presented group. There exists an infinite family of pairwise non-homeomorphic non-spin $4-$manifolds
 $M_i,i\in \NN,$ with fundamental group $G$ such that each
$M_i$ supports infinitely many distinct smooth structures $M_{i,j}, j \in \NN,$ which do not admit an Einstein metric. All $M_{i,j}$ satisfy Hitchin-Thorpe inequality.
\end{main}

The manifolds constructed in Theorems \ref{main}, \ref{non-spin+arb} are necessarily
 non-spin.  However, using obstructions   induced by a non-trivial Bauer-Furuta
 invariant \cite[Theorem D]{i-lb}, one can also find examples of spin manifolds which satisfy the Hitchin-Thorpe inequality,
 but do not admit Einstein metrics.

\begin{main}\label{main-spin}
For any finite group $G$, there is an
 infinite family of spin  $4-$manifolds $M_i$ with fundamental group
 $\pi_1(M_i)=G$ such that each space supports infinitely
 many distinct smooth structures, which do not admit an Einstein
 metric. All these manifolds satisfy the Hitchin-Thorpe
 inequality. 
\end{main}

For a small fundamental group  $\ZZ/d\ZZ$, 
 some of the manifolds constructed in the theorem support
 a differential structure which admits a K{\"a}hler-Einstein metric, 
 but most of the manifolds
 $M_i$ do not support a complex structure.

On simply connected spin manifolds it was proved that there are no
 Einstein metrics for some exotic smooth structures \cite{i-lb}. We  prove
that there are no  $\ZZ/d\ZZ-$invariant Einstein metrics for the standard
  smooth structure on certain connected sums of $K3$'s and 
  $S^2\times S^2$'s.

\begin{main}\label{non-ex-spin}
There exists an $n_0>0$ such that for any $d>n_0$ the manifolds:
\begin{enumerate}
\item$M_{1,n}=d(n+5)(K3)\# (d(n+7)-1)(S^2\times S^2)$ \\
\item $M_{2,n}=d(2n+5)(K3)\#  (d(2n+6)-1) (S^2\times S^2)$
\end{enumerate}
$n\in \NN^*,$ admit infinitely many inequivalent free $\ZZ/d\ZZ$-actions, 
such that there is no Einstein metric on $M_{1,n},M_{2,n}$ invariant under any of the $\ZZ/d\ZZ-$actions.
\end{main}

The paper is organized as follows: in Section \ref{top} we review some background results on the
 topology of $4$-manifolds, in the third section we give a
 construction of manifolds  which admit free $\ZZ/d\ZZ$-actions
 and in Section \ref{proofs} we give the proofs of the above theorems.

\bigskip
{\bf Acknowledgments.} 
I would like to thank  Professor Claude LeBrun for suggesting the problem in Theorem \ref{main} and for numerous discussions on the subject.

\section{Background results on the topology of 4-manifolds}\label{top}

\subsection{Homeomorphism criteria}
A remarkable result of Freedman \cite{freedman} and Donaldson \cite{don} tells us that
 smooth compact simply connected oriented $4$-manifolds are
 classified up to homeomorphism by their numerical invariants: Euler characteristic
 $\chi$, signature $\tau$ and Stiefel-Whitney class $w_2$. 
 In the non-simply connected case this was generalized in \cite[Theorem C]{ha-kr}:

\begin{thm}[Hambleton, Kreck]\label{hakr}
Let $M$ be a smooth closed oriented $4$-manifold with finite
 cyclic fundamental group. Then $M$ is classified up to
 homeomorphism by the fundamental group, the intersection form
 on $H_2(M,\ZZ)$ modulo torsion, and the $w_2$-type. Moreover, any isometry of
 the intersection form can be realized by a homeomorphism.
\end{thm}

In contrast with simply connected manifolds, there are three 
 $w_2$-types that can be exhibited: (I) $w_2(\widetilde M) \neq 0$,
 (II) $w_2(M)=0$, and (III) $w_2(\widetilde M)=0,$ but $w_2(M)\neq 0.$

Using Donaldson's and Minkowski-Hasse's classification of the
 intersection form we can reformulate this theorem on an easier
 form: 

\begin{rmk} We have the following equivalent statement:
a smooth, closed, oriented $4$-manifold with finite cyclic
 fundamental group and indefinite intersection form is
 classified up to homeomorphism by the fundamental group, the
 numbers $b_2^\pm,$ the parity of the intersection form and
 the $w_2$-type.
\end{rmk}

In the presence of 2-torsion one must be careful to determine both the
 parity of the intersection form and the $w_2$-type, as there are known examples of
 non-spin manifolds with even intersection form, see for example
 the Enriques surface.
On a 4-manifold with finite fundamental group, knowing
 the invariants $b_2^\pm$ is equivalent to knowing any other two
 numerical invariants, for example the Euler characteristic $\chi$ 
 and signature $\tau.$

\subsection{Almost complete decomposability}

Freedman's and Donaldson's results tell us that any non-spin simply connected 
$4$-manifold is homeomorphic to
 $n\bcp^2\# m\overline{\bcp^2},$ for appropriate positive
 integers $n,m.$ 
Using non-triviality of the Seiberg-Witten invariants, infinitely 
 many differential structures 
 can be exhibited on many manifolds which admit a smooth, symplectic
 structure. In the case of $b_2^+>1,$ such manifolds are never diffeomorphic to 
 $n\bcp^2\# m\overline{\bcp^2}$ as the Seiberg-Witten invariant of the later manifold vanishes. 
 One way of measuring how  different the differential structures are, 
 is by taking connected sum with a copy of $\bcp^2.$ In this case, 
 the Seiberg-Witten invariants of both manifolds vanish by Taubes' theorem. 
 Hence it is natural to introduce the following definitions, due to Mandelbaum and Moishezon:

\begin{defn}
 We say that a smooth, non-spin, $4$-manifold $M$ is {\em completely
 decomposable} if it is diffeomorphic to a connected sum of
 $\bcp^2$'s and $\overline{\bcp^2}$'s.
\end{defn}
\begin{defn}
 We say $M$ is {\em almost
 completely decomposable (ACD)} if $ M \# \bcp^2$ is completely
 decomposable.
\end{defn}
Almost completely decomposable manifolds are quite common. It has 
 been conjectured by Mandelbaum \cite{man} that simply connected
 analytic surfaces are almost completely decomposable. In support 
 of this conjecture he proved that any complex surface which is 
 diffeomorphic to a complete intersection of hypersurfaces in some 
 $\bcp ^N$ is almost completely decomposable. 
 Moreover, he also considered the branched covers:

\begin{prop}{\em \cite{man}}\label{branch}
Let $X \subset \bcp^N$ be a compact complex surface and suppose
 $M\longrightarrow X$ is an $r$-fold cyclic branched cover
  of $X$ whose branch locus is homeomorphic to $X \cap H_r$,
 for some hypersurface $H_r$ of degree $r$ of $\bcp^N$. Then if
 $X$ is almost completely decomposable so is $M$.
\end{prop}

For the definition of a $r$-fold cyclic cover see the next section.

Using iterated cyclic-covers of $\bcp^2$ and symplectic
 surgeries, Braungardt and
 Kotschick, \cite{bra-kot}, were able to show the abundance of the 
 ACD symplectic manifolds for fixed topological invariants:
 
\begin{thm}{\em \cite{bra-kot}}\label{geogr-sympl}
For every $\epsilon >0,$ there is a constant $c(\epsilon)>0$ such
 that every lattice point $(x,y)$ in the first quadrant satisfying
 $$ y\leq (9-\epsilon)x- c(\epsilon)$$
 is realized by the Chern invariants $(\chi_h, c_1^2)$ of
 infinitely many, pairwise non-diffeomorphic, simply connected,
 minimal, symplectic manifolds, all of which are almost completely
 decomposable.
\end{thm}

\section{A construction of surfaces of general type with finite cyclic
 fundamental group} \label{cover}

One technique which yields a large family of examples of complex
 manifolds is the finite cyclic branched cover construction.

 For more details on the constructions presented in this
 chapter we refer the reader to \cite{bpv} for an
 algebraic geometric point of view, or to \cite{go-st}
 for a more topological description.

\begin{defn} A (non-singular) {\em d-fold branched cover}
 consists of a triple $(X,Y,\pi)$, denoted by
 $\pi :X \rightarrow Y$, where $X,Y$ are connected
 compact smooth complex manifold and $\pi$ is a finite,
 generically $d:1$, surjective, proper, holomorphic
 map.
The critical set, $R\subset X,$ is called the {\em ramification
 divisor} of $\pi$ and its image $D=\pi(R)$ is called the {\em
 branch locus.} 
 \end{defn}

 For any point $y\in Y\setminus D$ there is a
 connected neighborhood $V_y$ with the property that
 $\pi ^{-1}(V_y)$ consists of $d$ disjoint subsets of $X$, each
 of which is mapped isomorphically onto $V_y$ by $\pi.$ 

\subsection{Cyclic branched covers}\label{cyclic-covers}
One special class of branched covers are the cyclic branched covers. 
They can be constructed as follows:

\begin{constr}\label{cyclic constr}{\em :}
Let $Y$ be a connected complex manifold and $D$ an effective
 divisor on $Y.$ Let $\OO(D)$ be the associated line bundle and
 $s_D \in \Gamma (Y, \OO_Y (D))$ the section vanishing exactly along $D$.
Suppose we have a line bundle $\cal L$ on $Y$ such that
 ${\cal O} _ Y (D)=\cal L ^{\otimes {\mathnormal d}}. $
We denote by $L$ the total space of $\cal L$
 and we let $p :L\rightarrow Y$ be the bundle projection. If
 $z\in \Gamma(L, p ^*\cal L)$ is the tautological section, then
 the zero divisor of $p^{*}s_D -z^d$ defines an analytic
 subspace $X$ in $L.$
 If $D$ is a smooth divisor then $X$ is
 smooth connected manifold and $\pi = p|_X$ exhibits $X$ as a
 d-fold ramified cover of $Y$ with branch locus $D.$ We call
 $(X,Y,\pi)$ the d-cyclic cover of $Y$ branch along $D,$
 determined by $\cal L.$
\end{constr}
 Given $D$ and $Y$, $X$ is uniquely
 determined by a choice of $\cal L.$ Hence, $X$ is uniquely defined if 
 $ Pic(Y)$ has no torsion. 
 
A {\em cyclic} branched cover is a d-fold cover such that
 $\pi_{|X\setminus R}:X\setminus R\to Y\setminus D$ is a (regular)
 cyclic covering. Hence it is determined by an epimorphism 
 $\pi_1(Y\setminus D)\to \ZZ_d,$ and $Y=X/ \ZZ_d.$ Moreover, a cyclic 
 $d$-cover is a Galois covering, meaning that the function field 
 embedding $\CC (Y) \subset \CC(X)$ induced by $\pi$ is a Galois 
 extension.

The following lemmas give us the main relations between the two manifolds:
\begin{lem} {\em \cite{bpv}} \label{cyc1}
Let $\pi:X \rightarrow Y$ be a d-cyclic cover of $Y$ branched
 along a smooth divisor $D$ and determined by $\cal L,$ where
 ${\cal L} ^{\otimes {\mathnormal d} }= {\cal O} _ Y(D).$ Let
 $R$ be the reduced divisor $\pi^{-1}(D)$ on $X.$ Then:
\begin{itemize}
\item [{\em (i)}] ${\cal O}_X(R)= \pi^{*}\cal L$;
\item [{\em (ii)}] $\pi^{*}[D]= d [R]$, in particular d is the
 branching order along $R;$
\item [{\em (iii)}] ${\cal K}_ X= \pi^{*}({\cal K}_ Y
\otimes {\cal L}^{d- 1}).$
\end{itemize}
\end{lem}

 As an immediate consequence, we are able to compute the relations
 between the topological invariants of $X$ and $Y$ in the case of
 complex surfaces:

\begin{lem}\label{cyc3}
Let $X,Y$complex surfaces and $\pi :X\rightarrow Y$ be as in Lemma \ref{cyc1}. Then:
\begin{itemize}
\item [{\em (i)}] $c_2(X)=d c_2(Y)-(d-1)\chi(D);$
\item [{\em (ii)}] $c_1^2(X)=d(c_1(Y)-(d-1)c_1({\cal L}))^2$
\end{itemize}
\end{lem}

In a more general set-up, we can define a d-cyclic branch cover 
 $\pi:X \rightarrow Y$ branched along a divisor with simple
 normal crossing singularities and $Y$ smooth manifold. In
 this case, $X$ will be a normal complex surface
 (\cite{bpv} I$.17.$) with singularities over
 the singular points of $D$.
 In order to understand the type of singularity, we have to consider
 a small neighborhood of a singular point of $D,$
  $U\subset Y.$
Let $(x,y)$ be local coordinates of $U$ such that $D$ is defined by the
 equation $xy=0.$ 
 Then $\pi^{-1}(0)$ is an isolated singularity of $X$ 
 and the open neighborhood of $\pi^{-1}(0)$, $\pi^{-1}(U)\subset X$ is 
 modeled in local coordinates by $z^d=xy \subset \CC^3.$ 
 This type of singularity is known as an $A_d-$singularity.

There are two techniques to associate a smooth manifold to $X$.
 The easiest method is to consider the resolution of the singularity
 (see for instance \cite[Chapter III, Theorems 6.1, 6.2]{bpv}).
 For any normal surface $X$ there exist bi-meromorphic maps
  $\pi  : X' \rightarrow X$, with $X'$ smooth, such that $\pi$ is an
  isomorphism outside the exceptional divisor $E=\pi^{-1}(Sing~ X)$.
  If $(X',X,\pi)$ are chosen such that $E$ does not contain any rational
  curves of self-intersection $(-1),$ then the minimal resolution is called
  {\it minimal}.
 The minimal resolution  is uniquely determined by the type of
 the singularities of $X.$
  A second method to eliminate singularities is given by deformations:
\begin{defn}\label{smoothing}
A smoothing of a normal surface $X$ is a proper flat map\linebreak
 $f:{\cal X}  \rightarrow \Delta,$ smooth over 
 $\Delta^*=\Delta \setminus {0},$ where $\cal X$ is a three
 dimensional complex manifold, $\Delta$ is a small open disk in
 $\CC$ centered at $0$, $ f^{-1}(0)$ is isomorphic to $X$ and $ f^{-1}(t)$ smooth if $t\neq0.$
\end{defn}

 If $t,t'\in \Delta^*$ then $f^{-1}(t)$ is diffeomorphic to
 $f^{-1}(t').$ In the case of cyclic branched coverings, a stronger
  result is true:
 
\begin{prop}\label{diffeom}
If $\pi :X\rightarrow Y$ is a $d-$cyclic cover branched along a
 divisor $D$ with simple normal crossing singularities, such
 that the linear system $\PP \big(H^0(Y,\OO (D))\big)$ is base
 point free. Then, there is a smoothing
 $\varpi: {\cal X} \rightarrow \Delta$ of $X$ and the
 generic fiber $X_t =\varpi^{-1}(t)$ is diffeomorphic to the minimal
 resolution $X'$ of $X.$
\end{prop}
\begin{proof}

We will give the complete proof for the case of double covers and then argue using 
the properties of local deformations of $A_d-$singularities \cite{hkk} 
that the same is true in general.

First, we remark that $X$ has a finite number of singular
 points, corresponding to the singular points of $D.$
 Resolving these singularities is a local process. The
 singularity, modeled by $(z^2=xy) \subset \CC^3$ is a
 quotient singularity. It is isomorphic to $\CC^2/ \ZZ_2$,
 where the $\ZZ_2$ action is given by the multiplication with $-1,$
 and the isomorphism is given by the map:
$$ \CC^2/ \ZZ_2 \to U,~~ \widehat{(u,v)} \mapsto (u^2, v^2, uv)=(x,y,z).$$
The minimal resolution of $\CC^2/ \ZZ_2$  can be constructed by
 blowing-up the origin of $\CC^2$ extending the $\ZZ_2$ action trivially
 on the exceptional divisor and then considering the new quotient space. 
 The total space of the blow-up of $\CC^2$ is the line bundle
 $\OO_{\bcp^1}(-1)$, and factoring by the above $\ZZ_2$ action
 corresponds to squaring (tensor product) the line bundle. 
 The resulting manifold, after taking the quotient, is
 $\OO_{\bcp^1}(-2).$ So, we resolved the singularities of
 $X$ by introducing an exceptional divisors which is a rational
 curve of self-intersection $(-2).$

Next, we explicitly construct a smoothing of $X.$ The idea
 is to consider a smoothing of the branch locus and construct the corresponding double cover.
 As the linear system $\PP \big(H^0(Y,\OO (D))\big)$ is
 base point free, there exists a holomorphic path of
 sections of $\OO(D)$ and a parametrization of this path given
 by $\varphi :\Delta \rightarrow \Gamma(Y, \OO(D))$ such that
 $\varphi (0)=\varphi_D$ and $ d\varphi_{|_{0,SingD}} \neq 0.$
 The last condition implies that the curve $\varphi (t)=0, t\neq 0$ sufficiently small, 
 doesn't contain any of the singularities of $D.$ Moreover, by Bertini Theorem we can also assume 
 that $\varphi_t:=\varphi(t), t\neq 0$
 corresponds to a smooth divisor.

Let ${\cal X} \subset {\cal L}\times \Delta$ given locally by the
 equation $ z^2 -\varphi (t)(x,y)=0,$ where ${\cal L} \to Y$ is a line
 bundle such that ${\cal L}^2=\OO(D)$ and $z$ is a local coordinate on the fiber of
 ${\cal L}.$ Then $\varpi: {\cal X} \rightarrow \Delta$  is a smoothing of 
 $X=\varpi^{-1}(0).$
 First, we can show that ${\cal X}$ is a smooth 
 manifold as: 
$$ d(z^2 -\varphi (t)(x,y))= 2zdz-\big(\frac{d}{dt}\varphi(t)(x,y)\big)dt-\big(\frac{d\varphi_t}
{dx}(x,y) dx +\frac{d\varphi_t}{dy}(x,y) dy\big)\neq 0$$
This is never zero as for $t\neq0$
 the section $\varphi_t$ is smooth, hence the last parenthesis is non-zero 
 and for $t=0$ we have $\frac{d\varphi}{dt}\mid_{0,SingD}\neq 0$ and $(\frac{d\varphi}{dx}+\frac{d\varphi}{dy})\mid _{0,D\setminus SingD}\neq 0$.
Then by Theorem $9.11$ in
 \cite{hartshorne} the morphism $\varpi $ is
 flat. 

The fact that the two constructions yield diffeomorphic
 manifolds is a local statement about the differential structures of
 the new manifolds in a neighborhood of the singularities. So, our proof is in local
 coordinates.  Because the morphism $\varpi $ is a submersion
 away from the central fiber it is enough to show that one of the
 fibers is diffeomorphic to $X'.$

In local coordinates the singularities are given by the
 equation
 $$(z^2-xy=0) \subset \CC^3.$$
 Because the linear system
 associated to $\OO(D)$ is base point free, then the zero
 locus of a generic section is smooth. We can consider a
 preferred local coordinates such that one of the smoothings, $X_1,$ is given by 
 $(z^2-xy=1) \subset \CC^3.$ If we change the local
 coordinates $(x,y,z)$ to  $(u,v,z),$ such that 
 $x=iu-v,~ y=iu+v$ then the smoothing is written in the
 canonical form $(z^2+u^2+v^2=1).$ Let $\xi=Re(u,v,z),
 ~\eta=Im(u,v,z), ~\xi,\eta\in \RR^3$ the real,
 imaginary part, of $(u,v,z).$ Then:\\
 $X_1= \{~(u,v,z)\in\CC^3~|~z^2+u^2+v^2=1~\}$\\
 $~~~~~=\{~(\xi,\eta)\in \RR^3 \times \RR^3~
 |~\parallel\xi\parallel^2-\parallel\eta\parallel^2 =1,
 <\xi,\eta>=0~ \}. $\\
The map $f:X_1 \rightarrow T^* S^2\subset \RR^3\times
 \RR^3$ defined by $f(\xi,\eta)=(\frac{\xi}{\parallel\xi\parallel},\parallel \xi \parallel\eta)$
 is a orientation preserving diffeomorphism.
 
  It is a well known fact
 that $\OO_{\bcp^1}(-2)$ is diffeomorphic to $T^* S^2,$ where
 exceptional divisor $\bcp^1$ is identified to the zero-section of the vector bundle.
 We have given an explicit construction for the case of double covers, as this sheds some
 light into the diffeomorphism between the minimal resolution and a smoothings, as well
 as emphasizes the different (local) holomorphic structures.

 To prove the result for a $d-$cyclic cover, the extra ingredient 
 needed is that the local smoothing of the singularities is diffeomorphic 
 to its minimal resolution. 
 For singularities of the form $z^d=xy$, or type $A_{d-1},$ the required diffeomorphism 
 is proved by Harer, Kas and Kirby, \cite{hkk},
by using topological Kirby calculus.
\end{proof}

A similar statement is true for double covers branched along a divisor 
 $D$ with {\em simple singularities}, i.e. the singularities are
 double or triple points with two or three different
 tangents or  simple triple points with one tangent. The
 double cover branched along such a divisor has $A-D-E$
 singularities, respectively. These singularities are called {\em
 rational double points}. It can be proved \cite{hkk} that for
 these singularities, and only for these singularities, the minimal
 resolution and the smoothing  are diffeomorphic.

Next we want to study the fundamental group of our manifolds. We need to introduce a new definition:

\begin{defn}
A smooth divisor $D$ is said to be flexible if there exists a
 divisor $D'\equiv D$ such that $D \bigcap D' \neq \emptyset$
 and $D'$ intersects $D$ transversally in codimension $2.$
\end{defn}

We remark that if $D$ is a flexible divisor then it must be connected.

Then, reformulating Catanese's Proposition 1.8 from \cite{fundgr} in our easy situation, we have:
\begin{prop}\label{s.c.covers}{\em \cite{fundgr}}
Let $\pi :X\rightarrow Y$ a d-cyclic cover branched along a
 smooth flexible divisor $D$, then if $Y$ is simply connected
 so is $X.$
\end{prop}

Notice that if the branch locus is not flexible then 
the covering manifold might not be simply connected: consider for 
example the double cover of $\bcp^1\times \bcp^1$ branched along 
four vertical $\bcp^1$. 

\subsection{Bi-cyclic covers of $\bcp^1\times\bcp^1$}

In this section we introduce a new construction which we are going to 
call the simple bi-cyclic cover of $\bcp^1\times \bcp^1$. Then, we study 
the analytical and topological properties of this class of manifolds.
 They are inspired by a construction due to Catanese. In his papers
 \cite{fundgr}, \cite{catconstr}, etc, he extensively studies two successive 
 double covers of $\bcp^1\times\bcp^1.$ For our purposes we need to consider
 cyclic covers of arbitrary degrees. We have the following construction:
\begin{constr}\label{const} Let $C,D$ be two smooth transversal curves in
 $\bcp^1\times \bcp^1$ such that the line bundles 
 $\OO (C), \OO (D)$ are d, p-tensor powers
 of some line bundles on $\bcp^1\times\bcp^1$. We construct a new manifold 
 $N$ by taking: first the d-cyclic cover of $\bcp^1\times\bcp^1$
 branched along $C$, which we denote by $X$, and then the p-cyclic cover of this manifold branched along 
 the proper transform of $D.$ We have the following diagram:
 $$N\to X\to \bcp^1\times\bcp^1$$
 It can easily be
 checked that if $C,D$ are smooth and intersect transversally,
 then both $X$ and $N$ are smooth. We call the manifold $N$ a
 {\em (simple) bi-cyclic cover of $\bcp^1\times \bcp^1$ of type $(d,p)$
 branched along $(C,D)$.}
\end{constr}

Let $\pi_1, \pi_2 :\bcp^1\times \bcp^1 \rightarrow \bcp^1$ be
 the  projections on the first, second factor, respectively.
 The line bundles on
 $\bcp^1\times\bcp^1$ are of the form $\pi_1 ^*(
 \OO_{\bcp^1}(a)) \otimes \pi_2^*(\OO_{\bcp^1} (b))$, which we
 denote by $\OO_{\bcp^1\times \bcp^1 } (a,b)$ or simply
 $\OO(a,b).$ With this notations, $\OO (C)=\OO (da,db)$
 and $\OO (D)=\OO (pm,pn)$ where $a,b,m,n$ are non-negative
 integers. Let $\varphi _C \in \Gamma(\OO
 (C))$ and $\varphi _D \in \Gamma(\OO (D))$ such that
 $C=\{\varphi_C =0\}$ and $D=\{\varphi_D =0\}.$ Let $E$ be the total space of ${\cal E}=\OO_{\bcp^1\times \bcp^1 } (a,b) \oplus \OO_{\bcp^1\times \bcp^1 } (m,n)$ and $\pi:E \to \bcp^1\times \bcp^1.$ Then the
 manifold $N$ can also be seen as a smooth compact submanifold of
 $E.$ If $(z,w)\in \Gamma(E,\pi^*\cal E)$ is the tautological section, then $N$  is defined
 by the following equations $\{ z^d-\pi^*\varphi_C =0,~ w^p- \pi^*\varphi _D=0\}.$
As an immediate consequence we obtain that Construction \ref{const} is commutative, i.e. we obtain the manifold $N$ also as the type $(p,d)$ bi-cyclic cover of $\bcp^1\times \bcp^1$ branched along $(D,C).$

Let $\pi $ be the projection from $N$ to $\bcp^1\times \bcp^1$
 induced by the fibration projection. Using
 Lemma \ref{cyc3}, we can easily compute the topological
 invariants of $N.$

\begin{lem}\label{bi-cyc-inv}
Let $\pi: N\rightarrow \bcp^1\times \bcp^1$ be a bi-cyclic cover as above. Then:
\begin{itemize}
\item [1.]$K_N=\pi^* \Big( \OO \Big(~(d-1)a+(p-1)m-2~,~(d-1)b+ (p-1)n-2 \Big) \Big); $
\item [2.]If $(d-1)a+(p-1)m \geq 3$ and $(d-1)b+ (p-1)n \geq 3$ then $K_N$ is an ample line bundle;
\item [3.]$c_1^2(N)=2pd\Big( (d-1)a+(p-1)m-2\Big)\Big((d-1)b+(p-1)n-2\Big);$
\item [4.]$c_2(N)=pd[4-2(d-1)(a+b-dab)-2(p-1)(m+n-pmn)+\\
 (p-1)(d-1)(an+bm)].$
\end{itemize}
\end{lem}
\begin{proof}
The proofs of $1,3$ and $4$ are an immediate consequence of Lemmas
 \ref{cyc1} and \ref{cyc3}. The extra ingredients that need to be
 computed are the Euler characteristics of the branch loci $C,D$,
 $\chi (C)= 2d((a+b)-dab)$, $\chi (D)= 2p((m+n)-pmn),$ and $D'\subset X$ the
 proper transform of $D.$ $D'$ is a $d-$cover of $D$ branched
 on card$(C\cap D)=dp(an+bm)$ points. Then:
 $\chi (D')=d\chi(D)- (d-1)(C\cdot D)=pd(2(m+n-pmn)-(d-1)(an+bm)).$

The canonical line bundle, $K_N,$ is the pull-back of an ample line bundle through a finite map 
 hence, see \cite{hartshorne}, $K_N$ is ample.
\end{proof}

Next we analyze the topological properties of our manifolds.
\begin{lem}\label{s.c.prop}
If $\pi: N\rightarrow \bcp^1\times \bcp^1$ be a bi-cyclic cover as above and $a,b,m,n$ are strictly positive integers, then $N$ is simply connected. 
\end{lem}
\begin{proof}
If $a,b,m,n$ are strictly positive integers, then the divisors
 $C,D$ are flexible divisors and  applying Proposition
 \ref{s.c.covers} twice gives us that the manifold $N$ is simply connected.
\end{proof}
As $N$ is a K{\"a}hler manifold it has non-trivial Seiberg-Witten invariants. Hence it does not decompose as connected sums of $\bcp^2$'s and $\cpb$'s. At most we can hope that it will decompose after taking the connected sum with a copy of $\bcp^2.$ We have a more general statement:

\begin{thm}\label{ACD}
Iterated cyclic covers of $\bcp^1 \times \bcp^1$ branched along 
 smooth, flexible curves, such that any two curves intersect transversally,
 are almost completely decomposable.
\end{thm}

\begin{proof}

We prove the theorem by double induction on the number of covers and the degree of the last cover.
 First we show that the
 $d$-cover $\pi_2:X\rightarrow \bcp^1\times \bcp^1$ branched along
 $C$ is almost completely decomposable. Let
 $\varphi_{(a,b)} :\bcp^1 \times \bcp^1 \rightarrow \bcp^P,$ where
 $P=(a+1)(b+1)-1,$ be the Segre-type  embedding, and $[p_{ij}]$ be
 homogeneous coordinates on $ \bcp^P$ corresponding to
 $(a,b)$-bi-homogeneous monomials. Then if $\varphi_C$ is the
 bi-degree $(da,db)$ polynomial whose zero locus is $C$, then
 $\varphi_C=\varphi_{(a,b)}^*(f(p_{ij}))$, where $f$ is a
 degree $d$ polynomial. Hence, by Proposition \ref{branch}, $X$
 is almost completely decomposable.

We need one more ingredient. The
 following lemma is proved in \cite{mand-moish}:
\begin{lem}{\em (\cite{mand-moish}, 3.4)}\label{ma-mo}
Suppose $W$ is a compact complex $3$-manifold and $V,X_1,X_2$ are
 closed simply-connected complex submanifolds with normal crossing in
 $W.$ Let $S=X_1\bigcap X_2$ and $C=V\bigcap S.$ Suppose as divisors
 $V$ is linearly equivalent to $X_1+X_2$ and that $C\neq \emptyset.$
 Set $n=\card C$ and $g$ be the genus of $S.$ Then we have the diffeomorphism:
  $$V~\# ~\bcp^2 \cong X_1~\#~X_2\#~ 2g~\bcp^2 ~\# ~(2g+n-1)
         \overline{\bcp^2}.$$
\end{lem}

For the last step of induction, let $\pi_2:X\to \bcp^1\times \bcp^1$ be an iterated cyclic cover and 
let $\pi_1:N\rightarrow X$ be the p-cyclic cover branched along
 $D'=\pi_2^{-1}(D).$ Then 
 $\OO(D')= \pi_2^*(\OO_{\bcp^1\times \bcp^1}(pm,pn))$ and let
 $\cal L=\pi_2^*(\OO_{\bcp^1\times \bcp^1}(m,n)).$ Assuming the notations
 from Construction \ref{cyclic constr}, there exists a tautological
 section $z\in \Gamma(L,p^*(\cal L))$ such that $N\subset L$ is
 the zero locus of
 $z^p-p^*(\varphi_{D'})\in \Gamma(L,p^*({\cal L}^{\otimes p})).$
We can compactify $L$ to $W=\PP(\cal L \oplus \OO_X)$ by adding a
 divisor $W_\infty =W\setminus L.$ Then
 $p':W\rightarrow X$ is a $\bcp^1$-bundle and we can try to extend
 the section $z$ to a section in $\Gamma (W,p'^*(\cal L)).$ As
 $p'^*(\cal L)$ restricts on the fiber to the trivial line bundle and $z$ has a zero at
 the origin of $L\rightarrow X$ then this section will have a pole
 of multiplicity one along $W_\infty.$ To adjust to this problem we consider the section 
 $z'\in \Gamma( W, p'^*({\cal L}) \otimes\OO(W_\infty)),~z'=
 z\cdot \varphi_{W_\infty}.$ Let
 $\varphi_{W_\infty}=\varphi_\infty.$ Then $z'$ has no zero
 along $W_\infty$ and rescales the values of $z$ outside
 $W_\infty.$ Hence, our manifold
 $N$ is the zero locus $(~z'^p- p'^*(\varphi_D)\cdot \varphi_\infty ^p=0~).$

Let $D_1,D_2$ be two smooth curves on $\bcp^1 \times \bcp^1$,
 such that
  $\OO(D_1)=\OO(m,n),\OO(D_2)=\OO((p-1)m,(p-1)n)$ and which are transversal
   to each other and to the branch loci which where considered in the 
   construction of $X$ and such that none of the intersection points of 
   $D_1,D_2$ lie on the previous branch loci.
  Then $p'^{-1}(D_i)=D'_i,~i=1,2$ are
 smooth curves on $X$ transversal to each other and
 $D'_1+D'_2$ is linearly equivalent to $D'.$ Let
 $$X_1=(z'-p'^*(\varphi_{D_1})\cdot \varphi_\infty=0)$$
 $$X_2=(z'^{p-1}-p'^*(\varphi_{D_2})\cdot \varphi_\infty^{p-1}=0)$$

 We remark that $X_1$ is a cover of $X$ of degree one, hence it is diffeomorphic 
 to $X,$ while $X_2$ is a $(p-1)-$cyclic cover of $X$. Then $N, X_1,X_2$ verify 
 the requirements in the lemma so:
$$N~\# ~\bcp^2 \cong ~X~\# ~X_2\# ~r\bcp^2~\# ~s\overline{\bcp^2},\textrm{ for suitable }~r,s.$$
 If $R=(\varphi_{D_1}^{p-1}-\varphi_{D_2}=0) \subset \bcp^1 \times \bcp^1.$  Then $X_1 \bigcap X_2=S=(\pi_2 \circ p')^{-1}(R)$ is an iterated cover of $R.$ After taking the first branch cover we obtain a curve which is the $d-$cover of $R$ branched at 
 $R\cdot C=dam+dbn$ points, so its Euler characteristic is $d ~\chi(R) -(d-1)d(am+bn) \leq 0.$
 Hence genus of $S$ is strictly greater than zero, i.e. $r\geq 1.$

If $p=2$ then both $X_{1,2}$ are diffeomorphic to $X$ hence almost completely decomposable, which implies $N$ almost completely decomposable by Lemma \ref{ma-mo}. The induction on degree of the cover finishes the proof.
\end{proof}

\subsection{Free actions of the finite cyclic group}

On the remaining part of this section we give a recipe for
 constructing complex 
 surfaces of general type with fundamental group isomorphic to
 $\ZZ/d\ZZ.$ They are quotients of a free $\ZZ/d\ZZ$ action on
 bi-cyclic covers 
 $\pi: N\rightarrow \bcp^1\times \bcp^1$ of type $(p,d)$. We
  construct the action of $\ZZ/d\ZZ$ explicitly.

 First we consider actions on an arbitrary line bundle $\OO(a,b).$\\
Let $x=[x_0:x_1], y=[y_0:y_1]$ be homogeneous 
 coordinates on $\bcp^1\times\bcp^1.$ We want to consider the
 action of $\ZZ_d=<e^{\frac{2\pi i}d}>$ on $\bcp^1\times\bcp^1$
  generated by:
$$ e^{\frac{2\pi i}d } * ([x_0:x_1],[y_0:y_1]) = 
([ e^{\frac{2\pi i }d}x_0:x_1],[e^{\frac{2\pi i }d }
 y_0:y_1]).$$
This action has four fixed points:
 $$([0:1],[0:1]),~([0:1],[1:0]),([1:0],[0:1]),~([1:0],[1:0]).$$

Let  $U_0=\{ [1:x_1]|x_1\in \CC \}$ and $U_1=\{ [x_0:1]|x_0\in \CC \}$
 be charts of $\bcp^1.$ Then the charts
 $ U_0\times U_0, U_0\times U_1,U_1\times U_0, U_1 \times U_1$
 form an atlas of $\bcp^1\times\bcp^1.$ The line bundle 
 $\OO (a,b)$ restricted to each of these charts admits a
 trivialization. Let $z_{00}, z_{01}, z_{10},z_{11}$ be the
 corresponding coordinates on each trivialization.

On the chart $U_1\times U_1\times \CC,$ we let the 
 $\ZZ_d =< e^{\frac{2\pi i }d }>$ action be generated by:
$$ 
e^{\frac{2\pi i}d } *([x_0:x_1],[y_0:y_1],z_{11}) \rightarrow  ([ e^{\frac{2\pi i }d }x_0:x_1],[e^{\frac{2\pi i }d }y_0:y_1],e^{\frac{2\pi i}d} z_{11}) 
$$
Using the change of coordinates, the above action  is
 generated in the other charts by:
\begin{align}
\mathrm{on}~~ U_0\times U_1: &~~ (x_1,y_0,z_{01})\to ( e^{\frac{2\pi i (d-1)}d}x_1,~e^{\frac{2\pi i }d }y_0,~e^{\frac{2\pi i(1+a)}d} z_{01}); \notag \\
\mathrm{on}~~ U_1\times U_0: &~~ (x_0,y_1,z_{10}) \rightarrow  ( e^{\frac{2\pi i }d}x_0,~e^{\frac{2\pi i (d-1)}d }y_1,~e^{\frac{2\pi i(1+b)}d} z_{10}); \notag \\
\mathrm {on}~~U_0\times U_0: &~~(x_1,y_1,z_{00}) \rightarrow  ( e^{\frac{2\pi i(d-1) }d}x_1,~e^{\frac{2\pi i (d-1)}d }y_1,~e^{\frac{2\pi i(1+a+b)}d} z_{00}). \notag
\end{align}

Hence, we have the following lemma:
\begin{lem}\label{freeact}
If each of the integers $ a+1, ~ b+1,~a+b+1$ is relatively prime to $d,$ then the above action of $\ZZ/d\ZZ$ on $\OO (a,b)$ is semi-free with four fixed points on the $0-$section.
\end{lem}

We also consider a weighted action of $\ZZ_d$ which
 is defined on $U_1\times U_1\times \CC$ as follows:
 $e^{\frac{2\pi i}d } *(x_0,y_0,z_{11}) \to (e^{\frac{2\pi i}d}x_0,e^{2\frac{2\pi i }d}y_0,e^{\frac{2\pi i}d}z_{11}).$
This action extends on the other coordinate charts as:
\begin{align}
\mathrm{on}~~ U_0\times U_1: &~~ (x_1,y_0,z_{01})\to ( e^{\frac{2\pi i (d-1)}d}x_1,~e^{2\frac{2\pi i }d }y_0,~e^{\frac{2\pi i(1+a)}d} z_{01}); \notag \\
\mathrm{on}~~ U_1\times U_0: &~~ (x_0,y_1,z_{10}) \rightarrow  ( e^{\frac{2\pi i }d}x_0,~e^{\frac{2\pi i (d-2)}d }y_1,~e^{\frac{2\pi i(1+2b)}d} z_{10}); \notag \\
\mathrm {on}~~U_0\times U_0: &~~(x_1,y_1,z_{00}) \rightarrow  ( e^{\frac{2\pi i(d-1) }d}x_1,~e^{\frac{2\pi i (d-2)}d }y_1,~e^{\frac{2\pi i(1+a+2b)}d} z_{00}). \notag
\end{align}

We have a similar lemma:
\begin{lem}\label{weightact}
If each of the integers $ a+1, ~ 2b+1,~a+2b+1$ is relatively prime to $d,$ then the above action of $\ZZ/d\ZZ$ on $\OO (a,b)$  is free outside the four points on the $0-$section.
\end{lem}

Requiring that a smooth flexible divisor
 $D\subset \bcp^1 \times \bcp^1$
 be invariant under the first $\ZZ/d\ZZ$ action is equivalent to
 $\varphi_D=0$ is  $\ZZ/d\ZZ$ invariant. But we need a stronger condition, we need to consider divisors $D$ such that $\varphi_D$ is a $\ZZ/d\ZZ$ invariant polynomial. Then $\varphi_D$  is a
 bi-homogeneous polynomial of bi-degrees divisible by $d.$ So,
 there are strictly positive integers $(a,b)$ such that
 $\OO(D)=\OO(da,db)$ and
 $$\varphi_D=\sum_{\begin{subarray}{1}
			i= \overline{0,a}\\
			j=\overline{0,b}
		    \end{subarray} }
a_{ij}X_0^{di}X_1^{d(a-i)}Y_0^{dj}Y_1^{d(b-j)}+
\sum_{i=\overline{0,da}}\sum_{\begin{subarray}{1}
			~~~~~j\\
			i\leq dj \leq db+i
		    \end{subarray}}
b_{ij}X_0^iX_1^{da-i}Y_0^{dj-i}Y_1^{d(b-j)+i}
$$
for some complex coefficients $a_{ij},b_{ij}.$
The linear system of $\ZZ/d\ZZ-$invariant sections of $\OO(da,db)$
 is base point free, hence by Bertini's Theorem the generic
 section is smooth, and
 we can choose $D$ such that none of the four points are on $D.$
 Hence $\ZZ/d\ZZ$ acts freely on $D.$

If we consider the weighted action then the class of $\ZZ/d\ZZ-$invariant
divisors might be larger then the one described above (if d even). But
this is not important from our point of interest.

Let $N'\rightarrow \bcp^1 \times \bcp^1$ a $d-$cyclic cover
 branched along $D$ with $\OO(D)=\OO (da,db)$  and such that
 $\varphi_D$ is $\ZZ/d\ZZ-$invariant. Then $N'$ is a
 submanifold in $\OO (a,b)$ and the
 $\ZZ/d\ZZ$-action on $\OO (a,b)$ restricts to $N'.$
 Moreover, if $D$ does not contain any of the fixed four
 points and the conditions in the Lemma \ref{freeact} are
 satisfied, then $\ZZ/d\ZZ$ acts freely on $N'.$

Unfortunately, for what we need, the above construction is not sufficient,
 and we need to consider bi-cyclic covers.
 
\begin{prop}\label{KE-manfds-constr}
Let $C,D$ smooth, flexible, transversal divisors on
 $\bcp^1\times\bcp^1$ such that both $\varphi_D$ and $\varphi_C$ are invariant under the $\ZZ/d\ZZ$ action and $\ZZ/d\ZZ$ acts freely on $D$. Let $a,b,m,n$ be positive integers such
 that $\OO(D)=\OO (da, db),\OO(C)=\OO(pdm,pdn)$, and $d$
 is relatively prime to each of the integers $a+1,b+1,a+b+1.$
Then, the bi-cyclic cover
 $\pi: N \rightarrow \bcp^1\times\bcp^1$ of type $(d,p)$
 branched along $(D,C)$ admits a free $\ZZ/d\ZZ$
 action. The quotient $M=N/(\ZZ/d\ZZ)$ has the following
 properties:
\begin{itemize}
\item[1.] $M$ is a smooth complex surface, with
 fundamental group $\pi_1(M) =\ZZ/d\ZZ;$ 
\item[2.] $K_{M}$ is an ample line bundle if 
$(d-1)a+d(p-1)m > 2$ and $(d-1)b+d(p-1)n >2;$
\item[3.] $c_1^2(M)=2p\Big( (d-1)a+d(p-1)m-2\Big)
\Big( (d-1)b+d(p-1)n-2\Big);$
\item[4.]
 $c_2(M)=p[4-2(d-1)(a+b-dab)-2d(p-1)(m+n-pdmn)+\linebreak
2(d-1)d(p-1)(an+bm)].$
\end{itemize}
\end{prop}

\begin{proof}
First, we need to define the $\ZZ/d\ZZ$
 action on $N.$ The action that we want to consider is the
 trivial extension of the action on $\OO(a,b)$ to an action
 on $\OO(a,b)\oplus \OO(dm,dn).$ The conditions
 from the theorem imply that the action restricts to $N$ as a
 free holomorphic action. Its quotient is a smooth complex
 surface, with fundamental group $\pi_1(M) =\ZZ/d\ZZ.$
 The numerical invariants of $N$ are described by Lemma
 \ref{bi-cyc-inv}. 
The invariants of $M$ are related to those of $N$ by the
 following relations: 
 $c_2(M)=\frac1d c_2(N), ~c_1^2(M)=\frac1d c_1^2(N).$ The 
 computations in $3.$ and $4.$ are immediate.

By Lemma \ref{bi-cyc-inv}, $K_N$ is ample which implies \cite{hartshorne} 
 that so is $K_M$.
\end{proof}

 The techniques developed in this section allow us to construct a
  family of manifolds with special topological properties:
\begin{prop}\label{constr-KE-mfds}
Given any integer $d\geq2,$ there are infinitely many complex surfaces of general type, $\{Z_i\}_{i\in \NN},$ with ample canonical line bundle, such that their fundamental groups $\pi_1(Z_i)=\ZZ/d\ZZ,$ they have $w_2-type ~I,$ odd intersection form, and $c_1^2(Z_i)<5\chi_h(Z_i).$ Moreover, their universal covers are almost completely decomposable manifolds. 
\end{prop}
\begin{proof}
 We need to consider two cases, determined by the parity of $d$.

For $d$ odd we have the following construction: Let $M(d;a,b,m,n)$ the manifold constructed in Proposition
 \ref{KE-manfds-constr} as a $\ZZ_d$-quotient of a bi-cyclic
 cover of $\bcp^1\times\bcp^1$ of  type $(d,2),$ branched along $(D,C)$ with
 $\OO(D)=\OO(da,db)$, $\OO(C)=(2dm,2dn).$ Let
 $Z_i=M(d;d,d,i,i).$ As $d+1,2d+1$ are relatively prime to $d,$ 
 the conditions in Proposition \ref{KE-manfds-constr} are
 satisfied. An easy computation of the numerical invariants of
 $Z_i$ yields:
\begin{itemize}
\item $c_1^2(Z_i)=4(d(d-1)+di-2)^2$;
\item $\chi_h(Z_i)=\frac 13 d^2(d-1)(2d-1)+d(d-1)(di-1)+d^2i^2-2di+2.$
\end{itemize}
We can compute the signature in terms of these invariants as\linebreak 
 $\tau(Z_i)=(-8\chi_h+c_1^2)(Z_i).$ 
 Rohlin's Theorem states that on a spin manifold we have the
 following relation: $\tau\equiv 0 \bmod 16$. For $i$ odd,
 $c_1^2(Z_i)\neq 0 \bmod 8$, hence $\tau(Z_i) \neq 0\bmod 8$. If $d$ odd, 
 then the first homotopy has no $2-$torsion hence $Z_i$ non-spin implies 
 odd intersection form ( \cite{gompf}5.7.6).
  Its universal cover $\widetilde{Z_i}$ has signature
 $\tau(\widetilde{Z_i})=d\tau(Z_i)$, so for $d,i$ odd numbers
 $ \tau(\widetilde{Z_i})\neq 0\bmod 16$. Hence $Z_i$ is of
 $w_2$-type (I) and odd intersection form.

As $i$ increases, $\frac{c_1^2}{\chi_h}(Z_i)$ approaches
 $\frac{4d^2 i^2}{d^2 i^2} =4.$

\smallskip

In the case $d$ even the above arguments do not work as $\tau(Z_i) =0\bmod 16$.
 We need a different construction. The
 idea is, though, the same. Let $\pi:N(d;a,b,m,m)\to
 \bcp^1\times\bcp^1$ be a bi-cyclic cover of type $(d,3)$
 branched along $(D,C)$ with
 $\OO(D)=(da,db),~\OO(C)=(3dm,3dm)$ and such that $D$
 intersects $C$ transversally. To simplify the notation we
 use  $N=N(d;a,b,m,m)$ whenever we want to prove a general
 statement about the whole class of manifolds.
Lemma \ref{bi-cyc-inv} tells us that the canonical line
 bundle is $K_N=\pi^*(\OO((d-1)a+2dm-2, (d-1)b+2dm-2)),$ hence
 $N$ is a surface of general type, with
 ample canonical line bundle if
$(d-1)a+2m-2 >0$, $(d-1)b+2m-2 >0. $ Using Lemma
 \ref{s.c.prop}, we can also conclude that $N$ is simply
 connected.
 
We want $N$ to be non-spin. We show that this is true if $d$
 even and $b$ odd,
 by finding a class $[A]\in H_2(N,\ZZ)$ such that
 $[A]\cdot w_2(N)\neq 0\bmod 2.$ We construct $N$ in two
 steps:
$$
N\xrightarrow{\pi_2}X\xrightarrow{\pi_1}\bcp^1\times\bcp^1
$$
 where first we consider a $3$-cyclic cover of
 $\bcp^1\times\bcp^1$ branched along $C$, and then a
 $d$-cyclic cover branched along $\pi_1^{-1}(D),$ the proper
 transform of $D$.

We construct a $1$-parameter family of deformations of $N$ in
 the following way. Let $D'=A\cup B\subset\bcp^1\times\bcp^1$
 such that $A=\{pt\}\times\bcp^1$, $B$ a smooth curve of 
 bi-degree $(da-1,db)$, and such that $A,B,C,D$ intersect
 transversally at all points. 
Let $N_0$ be the $d-$cover of $X$ branched along $\pi_1^{-1}(D').$ This is
  singular surface, with $A_{d-1}-$type singularities. To resolve
  these singularities we introduce strings of exceptional
  divisors, which we denote by $E_j$. We denote the minimal
  resolution of $N_0$ by $N'$. Proposition \ref{diffeom} tells us
  that $N'$ is a complex surface diffeomorphic to $N.$
 
We have the following diagram:
$$
N'\xrightarrow{\pi_0}N_0\xrightarrow{\pi'_2}X\xrightarrow{\pi_1}
\bcp^1\times\bcp^1
$$

 Let $A''\subset N'$ be the proper transform of $A'=(\pi_1\circ \pi'_2)
 ^{-1}(A)\subset N_0$. Lemma \ref{cyc1}
 tells us that:
\begin{align}
\OO(A'')= &~ \frac 1d (\pi_1\circ \pi'_2\circ \pi_0 )^*(\OO(1,0))+\Sigma a_i E_i, ~a_i\in\QQ,~ a_i\leq 0 \notag \\
K_{N'}= & ~(\pi_1\circ \pi'_2\circ \pi_0 )^*(\OO ((d-1)a+2dm-2, (d-1)b+2dm-2))\notag \\
 K_{N'}\cdot A'' = &~ c_1\Big( (\pi_1\circ \pi'_2\circ \pi_0 )^*(\OO ((d-1)a+2dm-2, (d-1)b+2dm-2))\Big)\cup \notag \\
&~~c_1\big( \frac 1d (\pi_1 \circ \pi'_2\circ \pi_0 )^*(\OO(1,0))+\Sigma a_i E_i\big)\notag \\
 = &~ \frac 1d 3d c_1\big( \OO ((d-1)a+2dm-2, (d-1)b+2dm-2))\big)\cup c_1(\OO(1,0)) \notag \\
= &~ 3((d-1)b+2dm-2) \notag \\
\equiv &~ 1\bmod 2 ~~\mathrm {if} ~b ~\mathrm {odd},~d ~\mathrm{even} \notag
\end{align}
Hence $N'$ is non-spin, and so is $N$.

Let $d=2^kd',~d'\in \NN$ odd be the decomposition of $d$.

We want to consider the manifolds:
$$N_i=N(d;2d',d',i,i)\subset \OO_{\bcp^1\times\bcp^1}(2d',d')\oplus\OO_{\bcp^1\times\bcp^1}(di,di)~ \mathrm {if}~d'\neq1,$$
$$\mathrm {and}~~N_i=N(d;6,3,i,i)~ \mathrm{if} ~~d=2^k.$$
On these manifolds, we can define a weighted $\ZZ/d\ZZ$ action as in
 Lemma \ref{weightact}, where we extend the action trivially on the
 second factor. For this action to be well-defined we need $\varphi_D$
 and $\varphi_C$ to be invariant under the induced action on
 $\bcp^1\times\bcp^1$. Such curves always exist.
 Moreover, we can choose $D,C$ such that $\ZZ/d\ZZ$ acts freely on
 them. 
The conditions in the Lemma \ref{weightact} are automatically satisfied by our
 choice of degrees. 

Let $Z_i=N_i/( \ZZ/d\ZZ).$ $Z_i$ is complex surface of general type, with ample canonical line bundle, finite fundamental group $\pi_1(Z_i)=\ZZ/d\ZZ$ and of $w_2$-type (I).
Its numerical invariants can be computed using Lemma \ref{bi-cyc-inv} to be:\\
$c_1^2(Z_i)=6\big( 2(d-1)d'+2i-2\big)\big( (d-1)d'+2i-2\big) $,\\
 $c_2(Z_i)=3[4-2(d-1)(3d'-2d'^2d)-4(2i-3i^2)+3(d-1)d'i],$\\
 $\tau(Z_i)=\frac13(c_1^2-2c_2)(Z_i)\equiv -6(d-1)d'i \bmod 4.$\\
For the special case $d=2^k$ the numerical invariants are computed by the same formulas, for $d'=3.$
From the last relation we see that if $i$ is odd, then $\tau(Z_i)\neq 0\bmod 8$ hence the intersection form is odd.
We take $Z_i$ to be the subsequence indexed by odd coefficients.
As $i$ increases, $\frac{c_1^2}{\chi_h}(Z_i)=\frac{12c_1^2}{c_1^2 +c_2}(Z_i)$ approaches $\frac{12\cdot 6\cdot4i^2}{6\cdot4i^2+3\cdot12i^2}= \frac{24}5=4.8 $
Hence considering both cases, $d$ odd or even, there is
 a constant $n_0 >0$ such that for any $i \geq n_0$ we have
 $c_1^2(Z_i) \leq 5 \chi_h(Z_i).$ We will re-index our sequence starting from $Z_{n_0}.$
  Moreover, the universal covers of $Z_i$ are almost completely decomposable as they are bi-cyclic covers of $\bcp^1\times\bcp^1$ (Theorem \ref{ACD}).
\end{proof}

\section{Proofs of Theorems}\label{proofs}

We can now prove the main theorems:

\begin{proof}[Proof of Theorem \ref{main}]

The manifolds $Z_i$ are the ones given by Proposition \ref{constr-KE-mfds}. They are complex surfaces
 of general type with ample canonical line bundle. Hence, 
 Aubin-Yau's Theorem on Calabi Conjecture,\cite{yau}, tells us that these manifolds admit
 K{\"a}hler-Einstein metrics. Moreover, they have odd intersection 
 form and $w_2$-type (I).

 By Theorem \ref{geogr-sympl}, there exist a constant $n_1>0$
 such that for any lattice point $(x,y)$ in the first quadrant
 verifying $x>n_1,~ y\leq 8.5x$  there exists a infinite
 family of homeomorphic, non-diffeomorphic simply connected
 minimal symplectic manifolds $M'_j$ such that
 $y=c_1^2(M'_j),~x=\chi _h(M'_j)=\frac{(c_1^2+c_2)(M'_j)}{12}.$
 Eventually after truncating and relabeling the sequence
 $Z_i$, we can construct $M'_{i,j},~i,j \in \NN$, a family
 of simply connected symplectic manifolds satisfying:
\begin{enumerate}
\item for fixed $i,~ M'_{i,j}$ are homeomorphic, but no two are
 diffeomorphic;
\item $\chi_h(M'_{i,j})=\chi_h(Z_i)$ for any $j\in \NN;$
\item $c_1^2(M'_{i,j}) \geq 8\chi_h(M'_{i,j});$
\item $c_1^2(Z_i)\leq 5\chi_h(Z_i).$
\end{enumerate}
Let $S_d$ be the rational homology sphere with
 fundamental group $\pi_1(S_d)=\ZZ/d\ZZ$ and universal cover $\#(d-1)S^2\times S^2$ constructed in \cite{ue}.

The manifolds $M_{i,j}$ are constructed as:
\begin{center}
 $M_{i,j}=M'_{i,j}\# S_d \#k \overline {\bcp^2},~~$ where
 $k=c_1^2(M'_{i,j}) -c_1^2(Z_i).$
\end{center}
We remark that by Theorem 1 \cite{k-m-t} the manifolds $M'_{i,j}\# S_d$
 are not symplectic, but they have non-trivial Seiberg-Witten
 invariant.

For fixed $i$ the manifolds $Z_i$ and $M_{i,j}$ are all
 of $w_2$-type (I), with odd intersection form,  fundamental group
 $\pi _1=\ZZ/d\ZZ$ and have the same numerical invariants:
 Todd-genus and Euler characteristic. Hence by Theorem
 \ref{hakr}, these manifolds are homeomorphic. 

In the construction theorem for $M'_{i,j}$ the differential
  structures were distinguished by different Seiberg-Witten 
  basic classes. After taking the connected sum with 
  $\overline{\bcp^2}$'s and $S_d$, and using Theorem 1 
  \cite{k-m-t} we see that the Seiberg-Witten basic classes 
  remain different. Hence the manifolds $M_{i,j}$ are not 
  diffeomorphic to one another.

An estimate of the number $k$ of copies of
 $\overline{\bcp^2}$ is given by:\\
$k=c_1^2(M_{i,j}) -c_1^2(Z_i)\geq 8\chi_h(Z_i)-5\chi_h(Z_i) =3\chi_h(Z_i)=3\chi_h(M_{i,j})$
We also know that the manifolds are under the Bogomolov-Miyaoka-Yau line, which implies:
$\chi_h(M_{i,j})~\geq~ \frac19 c_1^2(M_{i,j}).$\\
Hence $k ~\geq ~ \frac 13
 c_1^2(M_{i,j})=\frac 13 (2\chi+3\tau )(M_{i,j}).$ 
Then Theorem \ref{E-obstruction} implies that $M_{i,j}$ does not admit any
 Einstein metric. As a consequence we also get that $Z_i$ and
 $M_{i,j}$ are not diffeomorphic.

For the results from the second part of the theorem we have to look at
 the universal covers $\widetilde{Z_i},$ and $\widetilde{M_{i,j}}$
 respectively. From our construction, the universal cover of $Z_i$
 is a simply connected minimal
 complex surface of general type. It can not be
 diffeomorphic to the connected sum of $\bcp^2$ 's and $~\overline{\bcp^2}$'s as
 it has non-trivial Seiberg-Witten invariants, but Theorem
 \ref{ACD} tells us that after connected sum with one copy of
 $\bcp^2$ it decomposes completely.
The universal cover of $S_d$ is diffeomorphic to
 $(d-1)(S^2\times S^2)$, hence the manifold
 $\widetilde{M_{i,j}}\cong d M'_{i,j}\# dk \overline{\bcp^2}
 \#(d-1)(S^2\times S^2)$.
 But $(S^2\times S^2) \# \overline{\bcp^2}$ is the complex surface
 $\bcp^1\times \bcp^1$ blown-up at one point, which can also be
 presented as $\bcp^2\# 2 \overline{\bcp^2}$. So:
$$\widetilde{M_{i,j}}\cong d M'_{i,j}\# dk \overline{\bcp^2} \#(d-1)S^2\times
 S^2\cong d M'_{i,j}\#(d-1)\bcp^2\#(dk+d-1)\overline{\bcp^2}.$$
 The manifolds $M'_{i,j}$ are almost completely
 decomposable, hence $\widetilde{M_{i,j}}$ is diffeomorphic to
 the connected sums of a number of $\bcp^2$'s and
 $\overline{\bcp^2}$'s.
\end{proof}

\begin{proof}[Proof of Proposition \ref{explicit-2}]

Let $N$ be the double cover of $\bcp^2$ branched
 along a smooth divisor $D$, such that $\OO(D)=\OO_{\bcp^2}(8).$ Then $N$ is
 a simply connected, almost completely decomposable  surface of
 general type and by Lemma \ref{cyc3} its numerical invariants 
 are:
 $$c_2(N)=46,~c_1^2(N)=2,~\tau(N)=-30.$$
Hence $b_2^+=7$ and $b_2^-=37.$
By Theorem \ref{E-obstruction} the manifold $M=N \# \cpb \# S_2$
  does not admit any Einstein metric. 
Let $\wt{M}$ be the universal cover of $M$. Then we have the
 following diffeomorphisms:
$$\wt{M}\cong 2N \# 2\cpb \# (S^2\times S^2)\cong 2N\# \bcp^2
\# 3\cpb  \cong 15\bcp ^2\# 77\cpb.$$
As $M$ does not admit any Einstein metrics this implies that $\wt{M}$ does not admit
 any Einstein metrics invariant under the covering involution $\sigma$.

Let now $M'=N\# \cpb\#S_3$, where the manifold $N$ is the same as above and $S_3$ is a rational sphere.
Then the universal cover $\wt{M'}$ is diffeomorphic to
$$\wt{M'}\cong 3N \# 3\cpb \# 2(S^2\times S^2)\cong 3N\# 2\bcp^2
\# 5\cpb  \cong 23\bcp ^2\# 116\cpb.$$
\end{proof}

We can combine the idea of the above propositions with the geography of almost completely decomposable symplectic manifolds \ref{geogr-sympl} to prove a general result:

\begin{proof}[Proof of Theorem \ref{infit-act+metr}]
We only need to prove that the statement holds for manifolds in the first region defined in condition 2. 
as the second region is obtained from the first region by a change of orientation.

Let $\epsilon >0$ be a small positive number.
Theorem \ref{geogr-sympl} tells us that for any $\epsilon'>0$ there exists an 
$c(\epsilon')$ such that to any integer point $(x,y)$ in the latice:
 $$0< y\leq (9-\epsilon')x- c(\epsilon')$$ 
we can associate infinitely many homeomorphic, non-diffeomorphic almost completely decomposable 
minimal symplectic manifolds which have topological invariants $(\chi_h,c_1^2)=(x,y).$
Let $\epsilon'=\frac32 \epsilon$ then there exists an $c(\epsilon')$ with the above properties. Let 
$N(\epsilon)=\frac{2d}{3}(c(\epsilon')+1).$

Let $n,m$ integer points such that the conditions $1.$ and $2.$ in Theorem \ref{infit-act+metr} are satisfied. Condition $2.$ states the following:
$$n ~~< ~~ (6-\epsilon)m-N(\epsilon)$$
Or equivalently:
\begin{align}
\frac{n}{d} &~<~ (6-\epsilon)\frac{m}{d}- \frac{N(\epsilon)}{d}\notag\\
\frac{3n}{2d} & ~< ~ (9-\frac32\epsilon)\frac{m}{d}- \frac{3N(\epsilon)}{2d} \notag\\
\frac{3n}{2d} & ~< ~(9-\epsilon')\frac{m}{d}-c(\epsilon')-1 \notag\\
\frac{3n}{2d} +1 & ~< ~(9-\epsilon')\frac{m}{d}-c(\epsilon') \notag
\end{align}

Then $[\frac{3n}{2d}] +1$ and $\frac{m}{d}$ satisfy the conditions of Theorem \ref{geogr-sympl}, 
hence there are infinitely many symplectic manifolds $M_i$ such that $c_1^2(M_i)=[\frac{3n}{2d}] +1$ and $
\frac{\chi+\tau}{4}(M_i)=\frac{m}{d}$.
 Moreover $M_i$ are homeomorphic non-diffeomorphic almost completely decomposable manifolds. 
 The differential structures are distinguished by the difference in the Seiberg-Witten monopole classes.

Let:
$$X_i=M_i\#S_d\#k\cpb,~~k=[\frac{3n}{2d}] +1-\frac nd>\frac13 c_1^2(M_i)$$
Then $\frac{\chi+\tau}{4}(X_i)=\frac md$ and $(2\chi+3\tau)(X_i)=\frac{n}{d}.$ 
The manifolds $X_i$ remain homeomorphic to each other and  using the formula for the 
Seiberg-Witten monopole classes for the connected sum we can immediately see that any 
two manifolds are not diffeomorphic. 
Their universal cover $\wt{X_i}$ is diffeomorphic to $dX_i\#dk\cpb\#(d-1)(S^2\times S^2)$
 and as $X_i$ are almost completely decomposable, this implies that all $\wt{X_i}$ are diffeomorphic to 
 $X=a\bcp^2\#b\cpb$ for suitable $a,b$ such that $\frac{\chi+\tau}{4}(X)=m,(2\chi+3\tau)(X)=n.$

As $M_i\#S_d$ has non-trivial Seiberg-Witten invariants, 
we  use Theorem \ref{E-obstruction} to conclude that the manifolds $X_i$ do not admit an Einstein metric. 
Hence their universal cover $X$ does not admit any Einstein metric invariant under any of the 
$\ZZ/d\ZZ-$actions.
\end{proof}

\begin{proof}[Proof of Theorem \ref{non-spin+arb}]

The proof is again based on the obstruction given by Theorem \ref{E-obstruction}. 
To construct the examples we employ Gompf's techniques, \cite{gompf}, 
and use symplectic connected sum along symplectic submanifolds of self-intersection $0.$ 
We use some standard blocks of symplectic manifolds.

The first block, $X_G,$ was constructed by Gompf in \cite[Theorem 6.2]{gompf}.
It  has fundamental group $G$, $c_1^2(X_G)=0$ and contains a symplectic $2-$torus
 of self-intersection $0.$ 

The second block, denoted by $E(n),$ is the family of simply connected proper elliptic complex surfaces
 with no multiple fibers and Euler characteristic $\chi(E(n))=12n$. 
 A generic fiber is a symplectic torus of self-intersection $0$
  and it can easily be checked,\cite{go-st}, that its complement is simply connected.

The third block is the one used by Gompf \cite{gompf} in the proof
 of Theorem $6.1$. Let $F_1, F_2$ be two Riemann surfaces of genera
  $k+1~(k\geq1)$ and $2$, respectively. Let $C_i,~i=1,\dots,2k+2$ be
  homologically
 nontrivial embedded circles in $F_1$ generating $H_1(F_1,\ZZ)$ and
 the circles $C'_i \subset F_2 ,~i=1,\dots,4$ be the generators of
 $H_1(F_2,\ZZ),$ such that $C_{2i}\cap C_{2j}=C_{2i-1}\cap C_{2j-1}=\emptyset$ if $i\neq j$ and $C_{2i}\cdot C_{2j-1}=\delta_i^j,~i,j=\overline{1,k+1}.$ $C'_i$ are also satisfying the above conditions.
Let $T_i$ be the collection of tori given by
 $C_1\times C'_1,C_2\times C'_3, C_3\times C'_2,C_4\times
 C'_4,C_i\times C'_1,~i=5,\dots, 2k+2.$ We can perturb this
 collection to a new collection of {\em disjoint} tori $\{T'_i\}$,
 where $T'_i$ is homologous to $T_i$.
As $T_i \subset F_1\times F_2$ is a Lagrangian torus we may
 choose $T'_i \subset F_1\times F_2$ to be Lagrangian, too. Moreover these
 tori are also homologically non-trivial, hence we can perturb the product symplectic form on $F_1\times F_2$, see \cite{gompf}, such that these tori become symplectic
 submanifolds. Let $X_k$ be the manifold obtained by performing
 symplectic connected sum of $F_1\times F_2$ and $2k+2$ copies of
 $E(2)$ along the family $\{T'_i\}_{i=\overline{1,2k+2}}$ and
 generic fibers of $E(2)$.
Then the manifold $X_k$ is a spin, symplectic $4$-manifold, and by
 Seifert-Van Kampen Theorem it is also simply connected.
The numerical invariants of $X_k$ are $\chi(X_k)=52k+48,~ \tau(X_k)=-32(k+1).$

The fourth block has a linking role, and it is $E(4)$ with a special
 symplectic structure. This manifold has an important feature
 \cite[proof of Theorem $6.2$]{gompf}: it
 contains a symplectic $2$-torus and a disjoint symplectic surface of genus $2$. 
 We denote them by $T$ and $F,$
 respectively. Both $T$ and $F$ have self-intersection zero
 and their complement $E(4)\setminus \{F\cup T\}$ is simply
 connected.

Let:
$$M_i=X_G~ ~\#_{T^2}~ E(4)~\#_{\Sigma_2}~X_i ~\#_{\Sigma_2}~E(4)~\#_{T^2}~E(2)\#p \cpb,$$
where $\#_{T^2}$'s are the symplectic sums along tori of
 self-intersection zero and $\#_{\Sigma_2}$'s are fiber sums along Riemann surfaces of genus $2,$ represented by $F \subset E(4)$ and one generic $\{pt\} \times F_2 \subset X_i$.
  The last operation is simply a connected sum and $p$ is a constant satisfying: $$(c_1^2(X_i)+16)~>~p~> \frac 13 (c_1^2(X_i)+16)~>0.$$
The fundamental group of $M_i$ can be easily computed by
 Seifert-Van Kampen Theorem to be $G$.
To obtain different differential structures on $M_i$, we
 take logarithmic transformations of different multiplicities
 along a generic fiber of $E(2)$. Then by the gluing formula for the
 Seiberg-Witten invariants \cite{SW-glue} (Cor 15,20) the
 manifolds have different Seiberg-Witten invariants. Hence we have
 constructed infinitely many non-diffeomorphic manifolds. We
 denote them by $M_{i,j}$. 
For fixed $i$, these manifolds are
 all homeomorphic. To show this, we can first do the
 logarithmic transformations on $E(2)$, this yields
 homeomorphic manifolds. By taking the fiber sum along a
 generic fiber with the remaining terms we
 obtain homeomorphic manifolds.
Theorem \ref{E-obstruction} implies that no $M_{i,j}$ admits an
 Einstein metric. 
 \end{proof}
 
All the manifolds constructed up to now are non-spin. If we want 
 to analyze the spin case then we need to use a new
 obstruction obtained by LeBrun and Ishida \cite{i-lb}. 
 This uses a refinement of the Seiberg-Witten invariant defined
 independently by Bauer and Furuta.

\begin{thm} \label{spin-E-obstr}{\em \cite{i-lb}}
Let $X_j$, $j=1, \ldots, 4$ be  smooth, compact 
almost-complex $4$-manifolds for which the 
 mod-2 Seiberg-Witten invariant is non-zero, and 
suppose that
\begin{eqnarray}b_1(X_j)&=&0,  \\
b_2^+(X_j)&\equiv& 3 \bmod 4 , \\
{\textstyle \sum_{j=1}^4}b_2^+(X_j)&\equiv &4\bmod 8.  
\end{eqnarray}
Let $N$ be any oriented $4$-manifold
with $b_2^+=0$. Then, for any $m=2, 3$ or $4$, the smooth
 $4$-manifold 
$M= \#_{j=1}^mX_j\#N$ does not admit Einstein metrics if 
$$
 4m -   (2\chi + 3\tau )(N)    \geq \frac{1}{3}\sum_{j=1}^m
 c_1^2(X_j) .
$$
\end{thm}

\begin{proof}[Proof of Theorem \ref{main-spin}]

The manifolds $M_i$ are constructed as connected sums and fiber sums of different blocks.

For the first block we start with $X_G$, the spin symplectic 4-manifolds constructed by Gompf.
Taking the symplectic connected sum with the elliptic surface $E(2n)$ along arbitrary fibers ($\{pt\}\times T^2\cong F_0\subset E(2n)$) we obtain a new symplectic, spin manifold which we denote by $N_G(n)$. The elliptic surface $E(2n)$ that we consider has no multiple fibers and its numerical invariants are $c_2(E(2n))=24n,~\tau (E(2n))=-16n, ~c_1^2(E(2n))=0.$
Moreover the complement of the generic fiber $F_0$ is simply connected. Hence the manifold $N_G(n)$ satisfies the following: $\pi_1(N_G(n))=G, ~c_1^2(N_G(n))=0,~c_2(N_G(n))=24k+24n.$ Since $G$ is finite, $b_1(N_G(n))=0$, and $b_2^+(N_G(n))=4(k+n)-1\equiv 3\mod 4.$

The second block is obtained from $E(2)$
 after performing a logarithmic transformation of order $2n+1$ on one
 non-singular elliptic fiber. We denote the new manifolds by $Y_n$. All $Y_n$ are simply connected spin manifolds
 with $b_2^+=3$ and $b_2^-=19$, hence they are all homeomorphic. Moreover,
 $Y_n$ are K{\"a}hler manifolds and $c_1(Y_n)=2n f$, where $ f$ is the
 multiple fiber introduced by the logarithmic transformation (see \cite{bpv}). Hence $\pm 2nf$ is a basic class, and its Seiberg-Witten invariant is $\pm1$.

The third block is a "small" spin manifold, and we used it in the proof of the  previous theorem, $X_2$.
Its invariants are $c_2(X_2)=152, ~\tau(X_2)=-96,~ c_1^2(X_2)=16,~\mathrm{and}~b_2^+=27~(\equiv 3 \bmod4).$

We may now define our manifolds:
$$M_{i,j}=X_2~ \#~ N_G(i)~\#~ Y_j.$$
For fixed $i$, the manifolds $M_{i,j}$ are all homeomorphic 
as we take connected sums of homeomorphic manifolds.
 We denote this homeomorphism type by $M_i$.
If we consider the basic classes of the Bauer-Furuta invariant,
 then both $a=c_1(X_1)+c_1( N_G(i) )+c_1(Y_j)$ and
 $b=c_1(X_1)+c_1( N_G(i) )-c_1(Y_j)$ are basic classes. Then
 $4j ~| ~(a-b).$ But any manifold has a finite number of basic classes which are a diffeomorphism invariant. As we let $j$ take infinitely many values, this will imply that $M_{i,j}$ represent infinitely many types of diffeomorphism classes. For a more detailed explanation, we refer the reader to \cite{i-lb-o}.
 By Theorem \ref{spin-E-obstr} these manifolds do not support any Einstein metrics, but they satisfy Hitchin-Thorpe Inequality.
\end{proof}

\begin{proof}[Proof of Theorem \ref{non-ex-spin}]
We begin by constructing a simply connected, spin manifolds with small topological invariants and $b_2^+\equiv3\bmod 4.$ One such manifold is given by a smooth hypersurface of tridegree $(4,4,2)$ in $\bcp^1\times\bcp^1\times\bcp^1,$ which we denote by $X.$ Its numerical invariants can be easily computed to be $c_1^2(X)=16,~ c_2(X)=104,~b_2^+(X)=19$ and it is simply connected. Then by Freedman's Theorem, $X$ is homeomorphic to $4K3\#7(S^2\times S^2).$ A result of Wall \cite{wall} tells us that there exists an integer $n_0$ such that $X\# n_0(S^2\times S^2)$ becomes diffeomorphic to $4K3\#7(S^2\times S^2)\#n_0(S^2\times S^2).$

Assuming the notation from the previous theorem, let:
$$M_{1,n}^j= X\# Y_j\# E(2n)\#S_d$$
$$M_{2,n}^j= X\#E(2)\# Y_j\# E(2(2n-1))\#S_d$$
By Theoren \ref{spin-E-obstr} the above manifolds do not admit an Einstein metric.
 
The manifolds $\{M_{1,n}^j~|~j\in \NN\}$ are all homeomorphic, but they represent
 infinitely many differential structures. Moreover, if we consider 
 their universal cover, $\wt{M_{1,n}^j}$ is diffeomorphic to 
 $dX\# dY_j\# dE(2n)\#(d-1)(S^2\times S^2).$ 
 But Mandelbaum \cite{man}
  proved that both $Y_j$ and $E(2n)$ completely decompose as connected 
  sums of $K3'$s and $S^2\times S^2$'s after taking the connected sum 
  with one copy of $S^2\times S^2.$ Hence, for $d>n_0,$ the manifold
  $ \wt{M_{1,n}^j}$ is diffeomorphic to 
 $d(4K3\#7(S^2\times S^2))\# d(K3)\#d(nK3\#(n-1)
 (S^2\times S^2))\#(d-1)(S^2\times S^2)$ i.e. to
 $d(n+5)K3\#(d(n+7)-1)(S^2\times S^2)=M_{1,n}.$ 
 Notice that the diffeomorphism type of the universal cover does not depend on $j$.
  Hence on $M_{1,n}$ we have constructed infinitely many non-equivalent, free actions of 
  $\ZZ/d\ZZ,$ such that there is no Einstein metric which is invariant under any of the group actions. 
  But all $M_{1,n}$ satisfy the Hitchin-Thorpe inequality.
 The same arguments for the manifolds $M_{2,n}^j$ gives us the results for the second family of manifolds.
\end{proof}

\bibliographystyle{alpha} 

  \end{document}